\documentclass[11pt,amsfonts,amsmath]{article}
\usepackage{graphicx}
\usepackage{amssymb}
\usepackage{amsmath}
\usepackage{amsthm}

\newcommand{\N}{{\mathbb{N}}}
\newcommand{\R}{{\mathbb{R}}}
\newcommand{\C}{{\mathbb{C}}}

\newtheorem{Theorem}{Theorem}
\newtheorem{Proposition}{Proposition}
\newtheorem{Lemma}{Lemma}
\newtheorem{Remark}{Remark}
\newtheorem{Definition}{Definition}
\begin{document}


\title{HOPF BIFURCATION IN AN OSCILLATORY-EXCITABLE REACTION-DIFFUSION MODEL WITH SPATIAL HETEROGENEITY }

\author{Benjamin Ambrosio\\
Normandie Univ, UNIHAVRE, LMAH,  FR CNRS 3335, ISCN,\\76600 Le Havre, France. \\
Mail: benjamin.ambrosio@univ-lehavre.fr}

\maketitle

\begin{abstract}
We focus on the qualitative analysis of a reaction-diffusion with spatial heterogeneity.  The system is a generalization of the well known FitzHugh-Nagumo system in which the excitability parameter is space dependent. This heterogeneity allows to exhibit concomitant stationary and oscillatory phenomena. We  prove the existence of an Hopf bifurcation and determine an equation of the center-manifold in which the solution asymptotically evolves. Numerical simulations illustrate the phenomenon.
\end{abstract}

\section{Introduction}
\noindent
The following reaction-diffusion system of FitzHugh-Nagumo (FHN) type:
\begin{center}
\begin{equation}
\label{FHNRD}
 \left \{ 
     \begin{array}{rcl}
      \epsilon u_t&=& f(u)-v +d_u\Delta u,\\
      v_t&=& u-c(x,t) -\delta v+d_v\Delta v,\\
     \end{array}
      \right. 
\end{equation}
\end{center}
where $f(u)=-u^3+3u$, $\epsilon>0$ small, $\delta \geq 0$, $c(x)$  regular function, $d_u\geq 0$, $d_v\geq 0$, $d_ud_v\neq 0$,  and with Neumann Boundary (NBC) conditions on a regular bounded domain $\Omega$, is relevant for obtaining different kind of patterns and interesting phenomena in physiological context.  A  property of system \eqref{FHNRD} is that, due to the dependence of $c$ on space variable $x$, it can take advantage of both excitability and oscillatory regimes of the FHN system. Therefore, interesting phenomena can be obtained with this single Partial Differential Equation such  as spirals, mixed mode oscillations (MMO's), propagation of bursting oscillations, see \cite{Amb09}. Recall that the FitzHugh-Nagumo model, widely used in  mathematical neuroscience, is obtained by a reduction of the Hodgkin-Huxley model (4 equations) awarded  by the 1963 Nobel prize of Physiology and Medicine, see \cite{Fit61,HH,Nag} for original papers or for example  \cite{Izhik,Erm10} for good fundamental books. In this article, we focus on equation \eqref{FHNRD} in the case where  $c$ is only depending on $x$, $\delta=d_v=0, d_u=d$, and the space dimension is 1, i.e.:
\begin{center}
\begin{equation}
\label{FHNRD1D}
 \left \{ 
     \begin{array}{rcl}
      \epsilon u_t&=& f(u)-v +d u_{xx}\\
      v_t&=& u-c(x)\\
     \end{array}
      \right. 
\end{equation}
\end{center} 
on a real open interval $\Omega=(-a,a), a>0$ and with NBC $u'(-a)=u'(a)=0$. 
In order to understand the qualitative  behavior of system \eqref{FHNRD1D}, we must recall the behavior of the underlying ODE system:
\begin{equation}
\label{FHNdl}
 \left \{ 
     \begin{array}{rcl}
      \epsilon u_t&=& f(u)-v\\
      v_t&=& u-c.\\
     \end{array}
      \right. 
\end{equation}
We have for appropriate values of parameters, the following theorem, see \cite{Amb09-Th} and references therein, which is illustrated in figure \ref{figsolEDO}.
\begin{Theorem}
There exists a unique stationary point. If $|c|\geq 1$ the stationary point is globally asymptotically stable, whereas if $|c|<1$, it is unstable and there exists a unique limit-cycle that attracts all the non constant trajectories. Furthermore, at $|c|=1$, there is a supercritical Hopf bifurcation.  
\end{Theorem}
Another important feature of system  \eqref{FHNdl} is excitability: for  $|c|>1$ and $|c|$ not so far from $1$,  if a solution is taken away from a neighborhood of the stable point in a suitable direction,  it  undergoes trough a large oscillation before returning to its stable state.  This can be well understood by slow-fast analysis. Typical behaviors are represented in figure \ref{figsolEDO}.
\begin{center}
\begin{figure}[h!]
\includegraphics[scale=.5]{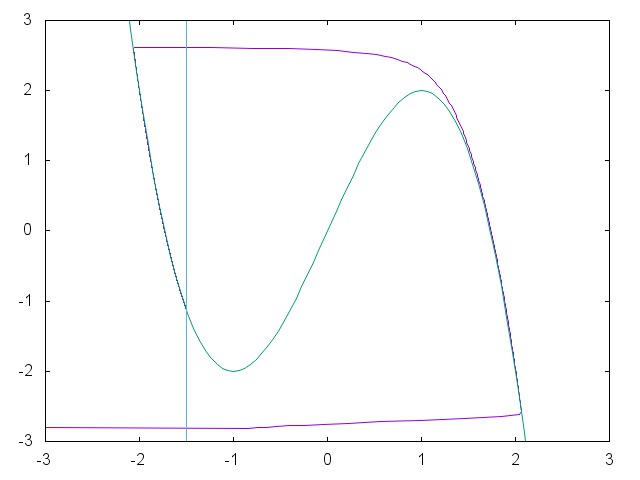}
\includegraphics[scale=.5]{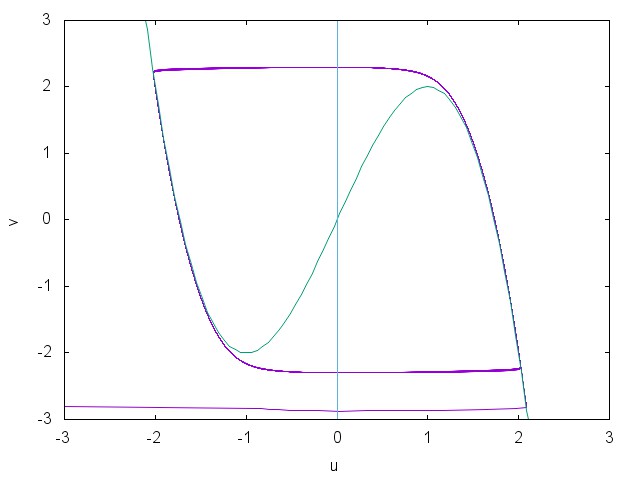}
\label{figsolEDO}
\caption{Solutions of system \eqref{FHNdl}, for typical values of $c$.The left panel illustrates excitability. The right panel illustrates oscillatory behavior.}
\end{figure}
\end{center}
Since the $c$ parameter is space dependent, we can couple oscillatory and excitable behavior via the diffusion term. For $x$ within a  central region, we choose $c(x)$ such that the system is in an oscillatory regime, whereas we choose  it excitatory anywhere else. We then address the question of wave propagation: will the center oscillations propagate along the domain trough excitability? We prove theoretically and show numerically that this depends on a parameter of excitability of the excitable cells. Varying this parameter, system \eqref{FHNRD1D} exhibits stable behavior or propagation of oscillations. This phenomenon occurs through an Hopf bifurcation in the infinite dimensional system \eqref{FHNRD1D}. Note that we have already exploited such an idea  in \cite{KAA} in the case of two coupled ODE slow-fast systems. The article is divided as follows:  we study the spectrum properties of the linearized system of \eqref{FHNRD1D} in the second section. In the third part, we apply the center manifold  theorem and compute restricted equations. Finally, in the fourth section we investigate numerically  the phenomenon.

\section{Hopf bifurcation for system \eqref{FHNRD1D}}
As in the case of ODE's, the linear stability analysis near the stationary solution gives some insights on the qualitative behavior of the system and allows to compute the equation for center manifold. Some theories have been developed,  see \cite{Carr82, Hen81, Kuz98}, however, the rigorous proofs in infinite dimensional  uses  strong theoretical background.  In this short paper, we will concentrate on the spectral properties of  the linearized operator and on the computation of the center manifold, leaving the more theoretical aspects for a forthcoming article.   In this section, we shall prove the existence of an Hopf bifurcation for system \eqref{FHNRD1D}. Some linear stability analysis for reaction-diffusion FitzHugh-Nagumo  systems has already been studied, see for example \cite{Cha74, Frei01, Rau78}, whereas a non-homogeneous FHN Reaction Diffusion system has been introduced in \cite{Dik05}. However, the following analysis, involving such a non-homogeneous space dependent term $c(x)$, is new.  After linearization near the stationary solution,  we  obtain an equation of regular Liouville type. We prove the positivity of an eigenvalue  for  small enough values of the bifurcation parameter by using classical spectral analysis. The remaining of the proof of the Hopf bifurcation, consists in proving that  an eigenvalue crosses the real axis as a parameter is varied. For this, we  introduce a polar change of coordinates. Then, the result follows from comparison theorems for ODE's.  We assume that the function $c(x)$, depending on a parameter $p>0$, is regular and satisfies the following conditions:
\begin{eqnarray}
c(x)\leq 0& \forall x \in (-a,a),\\
c(0)=0,\\
c'(x)>0& \forall x \in (-a,0),\, c'(x)<0 \, \forall x \in (0,a),\\
c'(-a)=c'(a)=0,\\
\forall x \in (-a,a), x \neq 0,&  c(x) \mbox{ is a decreasing function of }p, \\
\forall x \in (-a,a), x \neq 0,&\lim_{p\rightarrow 0}c(x)=0,\\
\forall x \in (-a,a), x \neq 0,&\lim_{p\rightarrow +\infty}c(x)=-\infty.
\end{eqnarray}
A typical function $c$ is for example: 
\[c(x)=p(\frac{x^4}{a^4}-2\frac{x^2}{a^2}).\]

Let  $X=C([-a,a],\R^2)$, endowed with the scalar product,
\[<(u_1,v_1),(u_2,v_2)>=\int_{-a}^au_1u_2dx+\int_{-a}^av_1v_2dx\]
It is a classical question that  equation \eqref{FHNRD1D} generates a dynamical system on $X\times X$.
Now, let us remark that the stationary solution is given by:
\begin{equation}
\label{FHNRD1Dsta}
 \left \{ 
     \begin{array}{rcl}
       \bar{v}&=& f(\bar{u})+d_u \bar{u}_{xx},\\
      \bar{u}&=& c(x).\\
     \end{array}
      \right. 
\end{equation}
The linearized system around $(\bar{u},\bar{v})$ is:  

\begin{equation}
\label{FHNLin}
 \left \{ 
     \begin{array}{rcl}
      \epsilon u_t&=& f'(\bar{u})u-v+du_{xx}, \\
      v_t&=& u.\\
     \end{array}
      \right. 
\end{equation}

We introduce the linear operator $\mathcal{F}$ with domain $\mathcal{D}(\mathcal{F})\{u,v \in C^2((-a,a)); u'(-a)=u'(a)=0\}$:

\begin{equation*}
\label{Op}
\mathcal{F}(u,v)=
 \left \{ 
     \begin{array}{l}
      \frac{1}{\epsilon}\big(f'(\bar{u})u-v+d u_{xx}\big), \\
       u.\\
     \end{array}
      \right. 
\end{equation*}

We proceed to the spectral analysis. We look for functions $u$, $v$ and numbers $\lambda$ such that:

\begin{equation*}
\label{eq:Vp}
 \left \{ 
     \begin{array}{rl}
      \frac{1}{\epsilon}\big(f'(\bar{u})u-v+d u_{xx}\big) &=\lambda u, \\
       u&=\lambda v,\\
     \end{array}
      \right. 
\end{equation*}
which is equivalent to
\begin{equation*}
\label{eq:Vp2}
 \left \{ 
     \begin{array}{rl}
 f'(\bar{u})u-\frac{u}{\lambda }+d u_{xx}&=\lambda\epsilon u, \\
       v&=\frac{u}{\lambda },\\
     \end{array}
      \right. 
\end{equation*}
or,
\begin{equation}
\label{eq:Vp3}
 \left \{ 
     \begin{array}{rl}
 -d u_{xx}-f'(\bar{u})u&=-\big(\frac{1}{\lambda }+\lambda\epsilon\big) u, \\
       v&=\frac{u}{\lambda }.\\
     \end{array}
      \right. 
\end{equation}

We set:
\[\nu=-(\frac{1}{\lambda }+\lambda\epsilon),\]
then the first equation writes,
\begin{equation}
\label{eq:SL}
-d u_{xx}-f'(\bar{u})u=\nu u, 
\end{equation}
and we have,
\begin{equation*}
\lambda^+_- =\frac{-\nu ^+_-\sqrt{\nu^2-4\epsilon}}{2\epsilon} .
\end{equation*} 

Note that equation \eqref{eq:SL} is a regular Sturm-Liouville problem.
We have the classical following theorem, see \cite{Tes10}, p 160-162.
\begin{Theorem}
\label{th:spec}
There exists an increasing sequence of real numbers $\nu_n$ and an orthogonal basis $(u_n)_{n\in \N}$ of $L^2(\Omega)$ such that:
\begin{equation}
\begin{array}{rcl}
\label{SL}
-du_{nxx}-f'(\bar{u})u_n&=&\nu_n u_n\\
u'(a)=u'(b)&=&0. 
\end{array}
\end{equation}
Furthermore, 
\[\lim_{n \rightarrow +\infty}\nu_n=+\infty,\]
and,
\begin{equation}
\label{SL}
\nu_0 = \inf_{u \in D(\mathcal{F});|u|_{L^2(\Omega)}=1} d\int_\Omega|\nabla u|^2dx-\int_\Omega f'(\bar{u})u^2dx. 
\end{equation} 
\end{Theorem}

We deduce the following proposition,
\begin{Proposition}
\label{Prop:solinst}
We assume that
\begin{equation}
\int_\Omega f'(\bar{u})dx>0,
\end{equation}
then at least one eigenvalue of $\mathcal{F}$ has a positive real part. 
\end{Proposition}
\begin{proof}
We consider $u=\frac{1}{\sqrt{|\Omega|}}$. Then,
\begin{equation*}
\nu_0 \leq -\int_\Omega f'(\bar{u})u^2dx<0,
\end{equation*} 
and
\begin{equation}
\label{eq:lam0}
{\lambda_{0}} ^+_- =\frac{-{\nu_0}^+_-\sqrt{\nu_0^2-4\epsilon}}{2\epsilon} 
\end{equation} 
has a positive real part. 
\end{proof}
\begin{Remark}
This proposition shows a result of instability which is directly linked with the stability of the ODE system \eqref{FHNdl}. Indeed, the assumption $|c|<1$ (instable steady state) corresponds to $f'(c)>0$ whereas  the assumption $|c|\geq 1$ corresponds to $f'(c)\leq 0$. For equation \eqref{FHNRD1D}, the asumption $f'(c)>0$ for instability is replaced by $\int_\Omega f'(\bar{u}(x))>0$.
\end{Remark}
Next, we prove that as $p$ decreases from $+\infty$ to $0$, the eigenvalue with the greatest real part crosses the imaginary axis from left to right, this proves the existence of the Hopf bifurcation. We start with the following lemma.
\begin{Lemma}
The equation \eqref{eq:SL}:
\begin{equation*}
-du_{xx}-(\nu+f'(\bar{u}))u=0,
\end{equation*}
rewrites
\begin{eqnarray}
\theta_x&=&g(\theta)=\cos^2\theta+\frac{f'(\bar{u})+\nu}{d}\sin^2\theta, \label{eq:theta}\\
r_x&=&\frac{\sin2\theta}{2} (1-\frac{\nu+f'(\bar{u})}{d})r, \label{eq:r}
\end{eqnarray}
with
 \begin{equation}
\label{chv}
u=r\sin\theta, u_x=r\cos\theta.
\end{equation}
\end{Lemma}
\begin{proof}
We have
\begin{equation*}
u_{x}=r_x\sin\theta+r\cos\theta\theta_x,
\end{equation*}
and
\begin{equation*}
u_{xx}=r_x\cos\theta-r\sin\theta\theta_x.
\end{equation*}

Multiplying the first equation by $\sin\theta$ and the second one by $\cos\theta$, adding both, we find:
\begin{equation*}
\begin{array}{rcl}
r_x&=&u_x\sin\theta+u_{xx}\cos\theta,\\
&=&r\sin\theta\cos\theta -\frac{f'(\bar{u})+\nu}{d}r\sin\theta\cos\theta, \\
&=&r\frac{\sin(2\theta)}{2}(1-\frac{\nu+f'(\bar{u})}{d}).
\end{array}
\end{equation*}
Multiplying the first equation by $\cos\theta$ and the second one by $-\sin\theta$, adding both, we find:
\begin{equation*}
\begin{array}{rcl}
r\theta_x&=&u_x\cos\theta-u_{xx}\sin\theta,\\
&=&r\cos^2\theta +\frac{f'(\bar{u})+\nu}{d}r\sin^2\theta,
\end{array}
\end{equation*}
which gives
\begin{equation*}
\theta_x=\cos^2\theta+\frac{f'(\bar{u})+\nu}{d}\sin^2\theta,
\end{equation*}
or equivalently
\begin{equation*}
\theta_x=1+(\frac{f'(\bar{u})+\nu}{d}-1)\sin^2\theta.
\end{equation*}
\end{proof}
In the following, we set
\[g(\theta)=\cos^2\theta+\frac{f'(\bar{u})+\nu}{d}\sin^2\theta.\]

The equation \eqref{eq:theta}  depends only upon $\theta$. Knowing $\theta$, equation \eqref{eq:r} gives:
\begin{equation*}
r(x)=r(-a)\exp{\int_{-a}^x\frac{\sin(2\theta(y))}{2} (1-\frac{\nu+f'(\bar{u}(y))}{d})dy}.
\end{equation*}

Therefore, we focus on the solutions of equation $\eqref{eq:theta}$. Since $u$ verifies NBC, we restrict ourselves to solutions with $\theta(-a)=\frac{\pi}{2}$ and $\theta(a)=\frac{\pi}{2} \mod \pi$. Note also that since $g(\theta)$ is $\pi$ periodic, if $\theta$ is solution also $\theta +n\pi$ is. It is therefore sufficient to consider initial conditions with $\theta(-a)=\frac{\pi}{2}$. Hence, we consider  solutions of \eqref{eq:theta} satisfying $\theta(-a)=\frac{\pi}{2}$, and we let a free boundary condition at $x=a$. Therefore, we obtain a Cauchy problem. Among all the solutions of the Cauchy problem, only those satisfying $\theta(a)=\frac{\pi}{2} \mod \pi$  correspond to eigenfunctions.
Next, we prove the following proposition which gives some qualitative behavior of $\theta$ for each $x \in [-a,a]$.
\begin{Proposition}
The function $\theta$ has the following properties:
\begin{eqnarray}
\forall x \in [-a,a],  \theta(x)>0,\\
\forall x\in (-a,a), \theta(x) \mbox{ is an increasing function of } \nu,\\
\forall x\in (-a,a), \lim_{\nu\rightarrow -\infty}\theta(x)=0,\\
\lim_{\nu\rightarrow +\infty}\theta(x)=+\infty.
\end{eqnarray}
\end{Proposition}
\begin{proof}
We have
\begin{equation*}
\theta_x=g(\theta)=1+(\frac{f'(\bar{u})+\nu}{d}-1)\sin^2\theta.
\end{equation*}
 Therefore,
 \begin{equation*}
g(0)=1>0.
\end{equation*}
 This implies that we have $\theta_x>0$ for $\theta=0$. Hence, since $\theta(-a)=\frac{\pi}{2}>0$, $\theta(x)$ cannot reach the value $0$ for $x\in [-a,a]$.  This proves the first claim.\\
 For the second claim, we use the theorem \ref{th:comp}. Assume that $\nu_2>\nu_1$, then  if 
 \begin{equation*}
\theta_{2x}=g_{\nu_2}(\theta_{2})= \cos^2\theta_2+\frac{f'(\bar{u})+\nu_2}{d}\sin^2\theta_2,
 \end{equation*}
 and 
  \begin{equation*}
\theta_{1x}=g_{\nu_1}(\theta_{1})= \cos^2\theta_1+\frac{f'(\bar{u})+\nu_1}{d}\sin^2\theta_1,
 \end{equation*}
 we have
  \begin{equation*}
  \begin{array}{rcl}
  \theta_{2x}&=& \cos^2\theta_2+\frac{f'(\bar{u})+\nu_1}{d}\sin^2(\theta_2)+\frac{\nu_2-\nu_1}{d}\sin^2(\theta_2),\\
  &=&g_{\nu_1}(\theta_{2})+\frac{\nu_2-\nu_1}{d}.
  \end{array}
 \end{equation*}
 This implies,
   \begin{equation*}
  \begin{array}{rcl}
  \theta_{2x}-g_{\nu_1}(\theta_{2})&\geq& 0,\\
  &=&\theta_{1x}-g_{\nu_1}(\theta_{1}).
  \end{array}
 \end{equation*}
Hence, by application of theorem \ref{th:comp}, we obtain 
\[\theta_{2}> \theta_1,\forall x \in (-a,a).\] 
 Let $k>\frac{f'(\bar{u})+\nu}{d}-1$. Then, if $\bar{\theta}$ is a solution of,
\begin{equation}
\label{eq:upper}
\bar{\theta}_x=1+k\sin^2\bar{\theta},
\end{equation} 
then $\bar{\theta}$ verify
\begin{equation*}
\begin{array}{rcl}
\bar{\theta}_x-g(\bar{\theta})&=&\bar{\theta}_x-(1+k\sin^2\bar{\theta})+(1+k\sin^2\bar{\theta}-g(\bar{\theta})\\
&=&(k-(\frac{f'(\bar{u})+\nu}{d}-1))\sin^2(\bar{\theta})\\
&\geq&0\\
&=&\theta_x-g(\theta).
\end{array}
\end{equation*}
Therefore, again by theorem  \ref{th:comp}, we obtain:
\[\bar{\theta}> \theta, \forall x \in (-a,a).\] 
 Now, for fixed $x\in (-a,a)$, and fixed $\gamma>0$ small, there exists $k$ such that $\bar{\theta}(x)<\gamma$. Then for  $\nu$ satisfying $\frac{f'(\bar{u})+\nu}{d}-1<k$ we have $0<\theta(x)<\bar{\theta}(x)<\gamma$.
Indeed, this follows by the following arguments. We can analyze the qualitative behavior of solution of \eqref{eq:upper}  which is a one dimensional autonomous ODE. 
We assume $k<0$, the  stationary point is a solution of:
\[\sin^2(\bar{\theta})=-\frac{1}{k}.\]
Let $\bar{\theta}^*$ the steady state solution of this last equation belonging to $(0,\frac{\pi}{2})$.
Since we start with $\bar{\theta}(-a)=\frac{\pi}{2}$, we have for $k<0$ small enough, $\bar{\theta}_x(-a)<0$, and this remains true as long as   $\bar{\theta}(x)> \bar{\theta}^*$. This means that  $\bar{\theta}(x)$ decreases and converges towards $\bar{\theta}^*$ as $x$ tends to $+\infty$. Furthermore, for fixed $x \in (-a,a)$, for fixed $\gamma$, one can choose $k<0$ small enough such that  $\bar{\theta}(x)\leq \gamma$. It is indeed sufficient to ensure for example
\[\frac{\pi}{2}+(1+k\sin^2(\gamma))(x+a)=\gamma.\]
  Since $\bar{\theta}(x)$ is an upper solution, this shows our third claim. The last claim follows from same arguments: for all $x \in (-a,a)$, for all $\gamma>0$ (large), there exists $k$, such that $\bar{\theta}(x)>\gamma$. Then for $\nu$ satisfying $\frac{f'(\bar{u})+\nu}{d}-1>k$ we have $0<\gamma<\bar{\theta}(x)<\theta(x)$.

\end{proof}

\begin{figure}
\begin{center}
\includegraphics[scale=0.2]{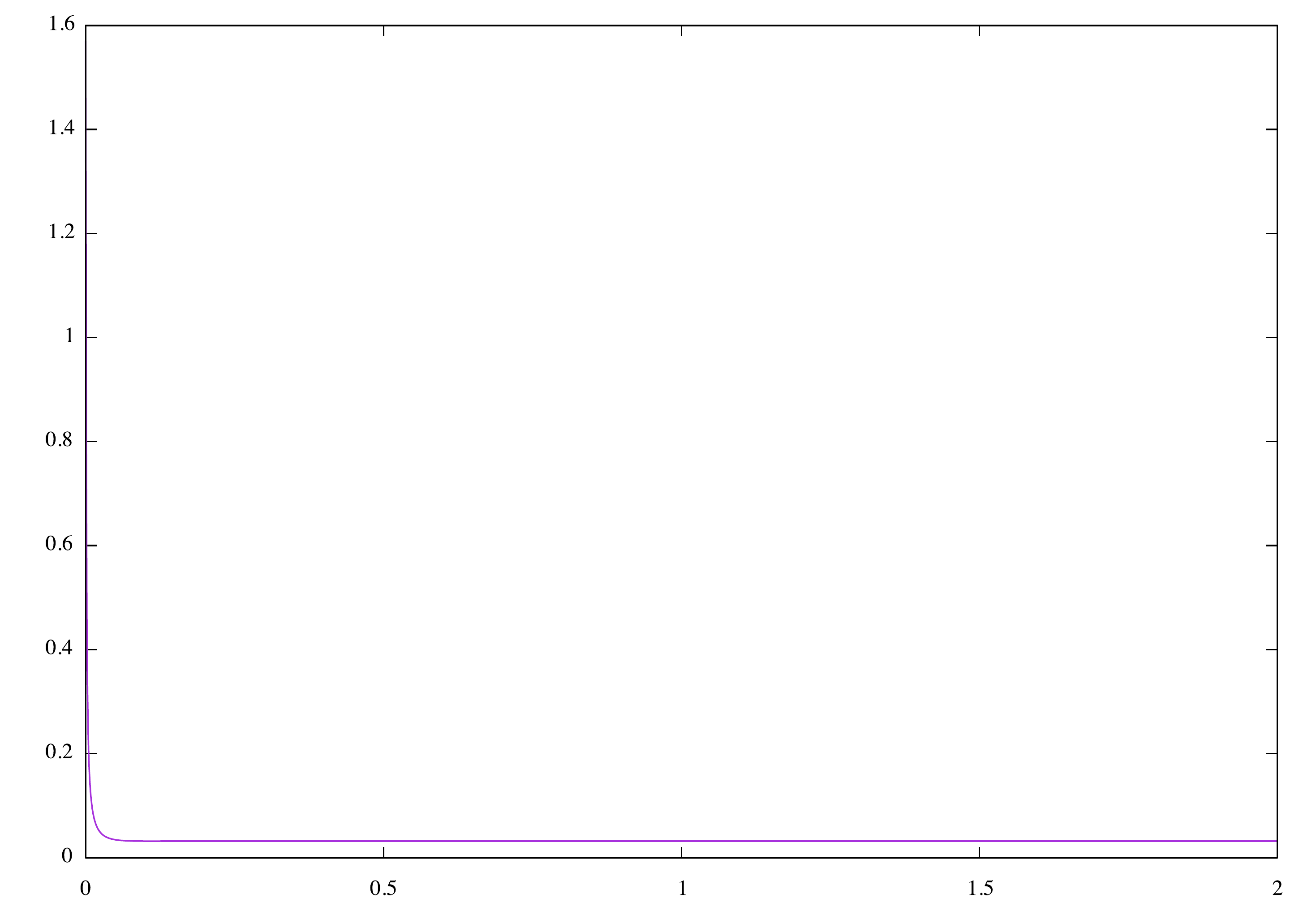}
\caption{Solution of equation \eqref{eq:upper} for $k=-1000$ and  $\bar{\theta}(0)=\frac{\pi}{2}$}
\end{center}
\end{figure}

We now will prove the following theorem,
\begin{Theorem}
For $p$ small enough, the linearized operator $\mathcal{F}$ has at least one eigenvalue with positive real part. For $p$ large enough, all the eigenvalues of the linearized system have negative real part. There is an Hopf Bifurcation:  there exists a value $p_0$ for which as $p$ crosses $p_0$ from right to left,  the real part of a conjugate complex eigenvalues increases from negative to positive. The other eigenvalues remaining with negative real parts.
\end{Theorem}
\begin{proof}
 We prove that:
\begin{enumerate}
\item for $p$ large enough, $\nu_0>0$, for $p$ small enough, $\nu_0<0$,
\item $\nu_0$ is an increasing function of $p$. 
\end{enumerate}
We start with the first step. \\
For $p$ small enough, i.e. close enough to $0$, $f'(0)=3>0$ implies  $f'(\bar{u}(x))> 0$ over $[-a,a]$ and the computation in the proof of proposition \ref{Prop:solinst} allows to conclude that $\nu_0<0$ in this case.
Now, we deal with equation \eqref{eq:theta}, with:
\[\theta(-a)=\frac{\pi}{2},\, \nu=0.\]
We prove that for $p$ large enough:
\begin{equation}
\label{eq:th(a)<pisur2}
\theta(a)<\frac{\pi}{2}.
\end{equation}
Since  $\theta$ is an increasing function of $\nu$, this implies that, for such a value of $p$, we have $\nu_0>0$.
In order to prove \eqref{eq:th(a)<pisur2}, we will find an upper solution $w$ of equation \eqref{eq:theta} such that: $w(a)<\frac{\pi}{2}$. By theorem \ref{th:uppersubsol} this implies that $\theta(a)<\frac{\pi}{2}$.
\begin{equation}
\label{eq:negfp}
\forall \mu>0, \forall \gamma<0, \exists p_0; \, p>p_0 \Rightarrow f'(\bar{u}(x))=f'(c(x))<\gamma \, \forall x \in [-a,-\mu]\cup[\mu,a], 
\end{equation}
and,
\begin{equation}
\label{eq:maxfp}
\frac{f'(\bar{u}(x))}{d}<\frac{3}{d} \, \forall x \in [-a,a].
\end{equation}

These two statements will allow us to construct the upper solution $w$. We first construct a piecewise linear function $w$ and then slightly modify it in order to have a $C^1$ function. The idea is to choose $w$ with a negative slope outside a small neighborhood of the origin, and with a positive slope within this neighborhood. We want the slopes to ensure that $0<w(a)<\frac{\pi}{2}$. Thanks to \eqref{eq:negfp} and \eqref{eq:maxfp},  this will ensure that $g(w)<w'$.
Let $\epsilon_1>0$ and let $w$ a continuous function such that $w(-a)=\frac{\pi}{2}$ and
\begin{equation*}
w(x)=\left\{
\begin{array}{rl}
\frac{\pi}{2}-\alpha (x+a)  &\mbox{ if } x \in (-a,-\epsilon_1),\\
\frac{\pi}{2}-\alpha (-\epsilon_1 +a)+ \frac{3}{d}(x+\epsilon_1) &\mbox{ si } x \in (-\epsilon_1,\epsilon_1),\\
\frac{\pi}{2}-\alpha (-\epsilon_1 +a)+ (1+\frac{3}{d})2\epsilon_1-\alpha (x-\epsilon_1) &\mbox{ if } x \in (\epsilon_1,a).\\
\end{array}
\right. 
\end{equation*} 
This means that $w$ is a continuous piecewise linear function with a negative slope $-\alpha$ for $x \in [-a,-\epsilon_1]\cup[\epsilon_1,a]$, and with a positive slope $\frac{3}{d}$ for $x \in [-\epsilon_1,\epsilon_1]$.
Moreover, we choose $\alpha$ and $\epsilon_1$ such that
\[\frac{\pi}{2}-2a\alpha >0,\]
which is equivalent to 
\[\alpha<\frac{\pi}{4a}.\]
This ensures that $w>0$ over $[-a,a]$.
Also, we choose,
\[-\alpha (-\epsilon_1 +a)+ \frac{3}{d}(2\epsilon_1)<0,\]
which is equivalent to:
\[\alpha > \frac{3}{d}\frac{2\epsilon_1}{a-\epsilon_1}.\]
This is always possible as soon as $\epsilon_1$ is small enough and ensures $w(x)<\frac{\pi}{2}$ over $(-a,a]$.

Then, in order to obtain a $C^1$ function, we slightly modify $w$, we set:
\begin{equation}
\begin{array}{c}
\tilde{w}(-a)  =w(-a),\\
\tilde{w}'(x)= w'(x) \mbox{ on } [-a,-\epsilon_1[\cup [-\epsilon_2,\epsilon_2] \cup [\epsilon_1,a], \epsilon_2<\epsilon_1,\\
\tilde{w}'(x)= -\alpha+\frac{\frac{3}{d}+\alpha}{-\epsilon_2+\epsilon_1}(x+\epsilon_1) \mbox{ on } [-\epsilon_1,-\epsilon_2],\\
\tilde{w}'(x)= -\alpha+\frac{\frac{3}{d}+\alpha}{\epsilon_2-\epsilon_1}(x-\epsilon_1) \mbox{ sur } [\epsilon_2,\epsilon_1],\\
\end{array}
\end{equation} 
with $\epsilon_2$ close enough to $\epsilon_1$.
For sake of simplicity, we rename $\tilde{w}, w$.
Then for $p$ large enough, we have:
\[w'>g(w).\]
Indeed, for $p$ large enough , $f'(\bar{u}(x))<0$ on $[-a,-\epsilon_2]\cup [\epsilon_2,a]$. Then, for all $x\in [-a,-\epsilon_2[ \cup [\epsilon_2,a]$:
\[g(w)<1+(\frac{f'(\bar{u}(x))}{d}-1)\inf_{x\in [-a,a]}\sin^2(w(x)).\]
This follows from the fact that $\sin^2(w(x))>\sin^2(w(a))>0$.
Then, for $p$ large enough,
\begin{equation}
\begin{array}{rl}
g(w)<-\alpha\leq w'  \mbox{ over }  [-a,-\epsilon_2[ \cup [\epsilon_2,a],\\
g(w)\leq \frac{3}{d}=w' \mbox{ over }  [-\epsilon_2,\epsilon_2]. \\
\end{array}
\end{equation} 
This shows that $w$ is an upper solution of \eqref{eq:theta}, therefore $\theta<w$. It follows that, $\theta(a)<w(a)<\frac{\pi}{2}$.
Therefore $\nu_0>0$ and all eigenvalues have a negative real part.
Now, we prove that $\nu_0$ is an increasing function of $p$. Since  $\theta(a)$ is an increasing function of $\nu$, it is sufficient to show that $\theta(a)$ is a decreasing function of $p$.  
Let $p_1>p_2$ and let us denote by $\theta_1, g_1$ (resp $\theta_2, g_2$) the solution and the $g$ function associated with $p_1$  (resp $p_2)$, we have:
\begin{align*}
\dot{\theta}_1-g_1(\theta_1)=&0,
\end{align*}
and
\begin{align*}
\dot{\theta}_2-g_1(\theta_2)&=\dot{\theta}_2-(\cos^2(\theta_2)+\frac{f'(\bar{u}_1)+\nu}{d}\sin^2(\theta_2))\\
&=\dot{\theta}_2-(\cos^2(\theta_2)+\frac{f'(\bar{u}_1)-f'(\bar{u}_2)+f'(\bar{u}_2)+\nu}{d}\sin^2(\theta_2))\\
&=-(\frac{f'(\bar{u}_1)-f'(\bar{u}_2))}{d}\sin^2(\theta_2))\\
&\geq 0.
\end{align*}
Therefore,
\[\dot{\theta}_1-g_1(\theta_1)\leq \dot{\theta}_2-g_1(\theta_2).\]
Furthermore,
\[\dot{\theta}_1(-a)<\dot{\theta}_2(-a)\]
which by theorem \ref{th:comp} implies that
\[\theta_2(x)>\theta_1(x) \mbox{ on } (-a,a],\]
This concludes the proof.
\end{proof}

\section{Application of the center manifold theorem}
In this section, in order to compute the equation for the center manifold, we formally apply the procedure described in \cite{Kuz98}. The theoretical analysis of the phenomenon using the framework of \cite{Carr82,Hen81,Lun95} is left for a forthcoming article. Let us rewrite \eqref{FHNRD1D} as
\begin{equation}
\label{eq:FG}
(u,v)_t=\mathcal{F}(u,v)+\mathcal{G}(u,v),
\end{equation}
where $\mathcal{F}$ is the linear operator defined in the previous part,

\begin{equation*}
\label{Op2}
\mathcal{F}(u,v)=
 \left \{ 
     \begin{array}{l}
      \frac{1}{\epsilon}\big(f'(\bar{u})u-v+d u_{xx}\big), \\
       u,\\
     \end{array}
      \right. 
\end{equation*}
and $\mathcal{G}$ is the nonlinear remaining part,
\begin{equation*}
\label{NL}
\mathcal{G}(u,v)=
 \left \{ 
     \begin{array}{l}
      \frac{1}{\epsilon}\big(\frac{1}{2}f''(\bar{u})u^2+\frac{1}{6}f''(\bar{u})u^3)= \frac{1}{\epsilon}(-u^3-3\bar{u}u^2)), \\
       0.\\
     \end{array}
      \right. 
\end{equation*}
 Let $\phi$ denote the dynamical system generated by equation \eqref{eq:FG}  on $X\times X$, and let $u_0$ the eigenfunction associated to $\lambda_0$ as defined in theorem \ref{th:spec}.
\begin{Theorem}
Let
\[T^c=u_0(x)Vect\{(1,0),(0,1)\}.\]
There is a locally defined
smooth two-dimensional  invariant manifold $W^c\subset X\times X$ that is tangent
to $T^c$ at  $0$.
Moreover, there is a neighborhood $U$ of $(0,0)$ in $X\times X$, such that if $\phi(t)(u,v) \in U$ for all $t\geq 0$, then $\phi(t)(u,v)\rightarrow W^c$
as $t \rightarrow +\infty$. The equation on the manifold can be approximated  by the complex equation 
\begin{equation}
\label{eq:z2}
z_t=\displaystyle \lambda_1z-\big{(}\frac{3}{C}\int_{\Omega}\bar{u}u_0^3dx\big{)}(z+\bar{z})^2-\frac{3}{C}\big{(}\int_{\Omega}u_0^2\bar{u}w_{20}^1dx+\int_{\Omega}u_0^4dx\big{)}z^2\bar{z}+...
\end{equation}
where $C=2\epsilon\int_\Omega u_0^2dx$,
whereas the first Lyapunov coefficient of the Hopf bifurcation is given by:
\begin{align*}
l_1(0)&=-\frac{3\sqrt{\epsilon}}{2C}(\int_\Omega u_0^4+\int_\Omega\bar{u}u_0^2Re(w^1_{20})),
\end{align*}
with 
\begin{equation*}
(\frac{3}{2}i\sqrt{\epsilon}-f'(\bar{u}))w^1_{20}-(w^1_{20})_{xx}=-6
\bar{u}u_0^2+12\frac{\epsilon}{C}u_0 \int_\Omega \bar{u}u_0^3dx.
\end{equation*}
\end{Theorem}
\begin{proof}
We define on the complexification of $X \times X$,  the following scalar product:
\[((u_1,v_1),(u_2,v_2))=\epsilon\int_{\Omega}\bar{u}_1u_2+\int_{\Omega}\bar{v}_1v_2.\]
Then,  after a simple computation, we find that the adjoint operator $\mathcal{F}^t$ of  $\mathcal{F}$ is given by:
\begin{center}
\begin{equation}
\label{Opt}
\mathcal{F}^t(u,v)=
 \left \{ 
     \begin{array}{l}
      \frac{1}{\epsilon}\big(f'(\bar{u})u+v+du_{xx}\big), \\
       -u.\\
     \end{array}
      \right. 
\end{equation}
\end{center}
At the Hopf birfucation parameter value $p=p_0$, the operator $\mathcal{F}$ has two purely complex conjugate eigenvalues $\lambda_1$ and $\lambda_2$, the others being of negative real part.
By putting $\nu_0=0$ in \eqref{eq:lam0}, we obtain:
\[\lambda_1=\frac{i}{\sqrt{\epsilon}}, \ \lambda_2=-\frac{i}{\sqrt{\epsilon}}. \]
Let us denote by $q$ the  eigenvector associated with $\lambda_1$. It follows easily from computation before theorem \ref{th:spec} that we can choose the first component of $q$,  $q_1(x)=u_0(x)$. Then, from \eqref{eq:Vp3}, we deduce that  its second component $q_2(x)=\frac{u_0(x)}{\lambda_1}=-i\sqrt{\epsilon}u_0(x)$. Therefore, we have, 
\[q(x)=u_0(x)\begin{pmatrix}
1\\
-i\sqrt{\epsilon}
\end{pmatrix}.\]
By computation, we see that the eigenvalues of $\mathcal{F}^t$ are the same as those of $\mathcal{F}$. Let $\tilde{p}$ the eigenvector of $\mathcal{A}^t$ associated to $\lambda_2$. After a short computation, we find:
\[\tilde{p}(x)=q(x).\]
Furthermore,
\[(\tilde{p},q)=2\epsilon\int_\Omega u_0^2dx.\]
Let
\[p=\frac{1}{2\epsilon\int_{\Omega}u_0^2dx}\tilde{p}.\]
Then:
\[(p,q)=1.\]
Let
\[T^c=u_0(x)Vect\{(1,0),(0,1)\}=u_0(x)Vect\{re(q),im(q)\},\]
and let
\[T^{su}\mbox{ be the orthogonal space of } T^c \mbox{ in } X\times X.\]
Let $\xi=(u,v)$. We set:
\[\xi=zq+\bar{z}\bar{q}+y\]
with $y\in T^{su}$. Then, we can verify that $zq+\bar{z}\bar{q}$ is the orthogonal projection of $\xi$ on $T^c$, and $z$, $\bar{z}$ are unique.
We also verify by computation that:
\begin{equation}
\label{eq:pqbar=0}
(p,\bar{q})=0.
\end{equation}
This implies that:
\begin{equation}
\label{eq:py=0}
y\in T^{su} \Leftrightarrow (p,y)=0.
\end{equation}
Indeed, 
\begin{equation*}
\begin{array}{rcl}
y\in T^{su}&\Leftrightarrow & \forall z \in \C, (zq+\bar{z}\bar{q},y)=0,\\
&\Leftrightarrow & \forall z \in \C,  (zq,y)=0,\\
&\Leftrightarrow &\forall z \in \C, (zp,y)=0,\\
&\Leftrightarrow &(p,y)=0.
\end{array}
\end{equation*}
It follows that:
\begin{equation}
\label{eq:z=y=}
\left\{\begin{array}{rcl}
z&=&(p,\xi),\\
y&=&\xi-(p,\xi)q-(\bar{p},\xi)\bar{q}.
\end{array}
\right.
\end{equation}
We derivate equation \eqref{eq:z=y=} with respect to time. We obtain,
\begin{equation}
\label{eq:zt=yt=}
\left\{\begin{array}{rcl}
z_t&=&(p,\xi_t),\\
y&=&\xi_t-(p,\xi_t)q-(\bar{p},\xi_t)\bar{q}.
\end{array}
\right.
\end{equation}
Using the fact that:
\begin{equation}
\xi_t=\lambda_1zq+\bar{\lambda}_1\bar{z}\bar{q}+\mathcal{G}(\xi),
\end{equation}
and by \eqref{eq:pqbar=0}, \eqref{eq:py=0}, we obtain after some computations that,
\[\left\{\begin{array}{rcl}
z_t&=&\lambda_1z+(p,\mathcal(G)(zq+\bar{z}\bar{q}+y)),\\
y_t&=&\mathcal{F}(y)+\mathcal{G}(zq+\bar{z}\bar{q}+y)-(p,\mathcal(G)(zq+\bar{z}\bar{q}+y))q-(\bar{p},\mathcal(G)(zq+\bar{z}\bar{q}+y))\bar{q},
\end{array}
\right.
\]
where we have used that $\mathcal{F}(y)\in T^{su}$.
Here, we have,
\begin{equation*}
\begin{array}{rcl}
\mathcal{G}(zq+\bar{z}\bar{q}+y)&=&\begin{pmatrix}
-\frac{1}{\epsilon}(3\bar{u}(zq_1+\bar{z}\bar{q}_1+y_1)^2+(zq_1+\bar{z}\bar{q}_1+y_1)^3)\\
0
\end{pmatrix}\\
&=&\begin{pmatrix}
-\frac{1}{\epsilon}(3\bar{u}((z+\bar{z})u_0+y_1)^2+((z+\bar{z})u_0+y_1)^3)\\
0
\end{pmatrix}.
\end{array}
\end{equation*}
Therefore,
\begin{equation*}
(p,\mathcal{G}(zq+\bar{z}\bar{q}+y))=-\frac{1}{C}\int_{\Omega}u_0(3\bar{u}((z+\bar{z})u_0+y_1)^2+((z+\bar{z})u_0+y_1)^3)dx,
\end{equation*}
where 
\begin{equation*}
C=2\epsilon\int_\Omega u_0^2dx.
\end{equation*}
Hence, we obtain,
\begin{equation}
\label{eq:zt}
z_t=\lambda_1z-\frac{1}{C}\int_{\Omega}u_0(3\bar{u}((z+\bar{z})u_0+y_1)^2+((z+\bar{z})u_0+y_1)^3)dx.
\end{equation}
Also, since the second coordinate of $\mathcal{G}$ is zero, we have,
\begin{equation*}
(\bar{p},\mathcal{G}(zq+\bar{z}\bar{q}+y))=(p,\mathcal{G}(zq+\bar{z}\bar{q}+y)).
\end{equation*}
Therefore,
\begin{equation*}
\begin{array}{rcl}
(p,\mathcal{G}(zq+\bar{z}\bar{q}+y))q+(\bar{p},\mathcal{G}(zq+\bar{z}\bar{q}+y))\bar{q}&=&(p,\mathcal{G}(zq+\bar{z}\bar{q}+y))(q+\bar{q}),\\
&=&(p,\mathcal{G}(zq+\bar{z}\bar{q}+y))\begin{pmatrix}
2u_0\\
0
\end{pmatrix},
\end{array}
\end{equation*}
from which we deduce that,
\begin{equation}
\label{eq:yt}
\begin{array}{rcl}
y_t&=&\displaystyle \mathcal{F}(y)+\begin{pmatrix}
-\frac{1}{\epsilon}(3\bar{u}((z+\bar{z})u_0+y_1)^2+((z+\bar{z})u_0+y_1)^3)\\
0
\end{pmatrix}\\
& &+\displaystyle \-\frac{1}{C}\int_{\Omega}u_0(3\bar{u}((z+\bar{z})u_0+y_1)^2+((z+\bar{z})u_0+y_1)^3)dx
\begin{pmatrix}
2u_0\\
0
\end{pmatrix}.
\end{array}
\end{equation}
It follows from the center manifold theorem, see \cite{Kuz98}, that we can write,
\begin{equation}
\label{eq:vc}
y=\frac{w_{20}}{2}z^2+w_{11}z\bar{z}+\frac{w_{02}}{2}\bar{z}^2+O(|z|^3),
\end{equation}
where  $w_{20},w_{11},w_{02}$, are to be determined later.
Following the procedure in \cite{Kuz98},  for the nonlinear terms in equation \eqref{eq:zt},
 we only write those with $(z+\bar{z})^2$ and $z^2\bar{z}$. We obtain,
\begin{equation}
\label{eq:z}
\begin{array}{rcl}
z_t&=&\displaystyle \lambda_1z-\frac{1}{C}\int_{\Omega}u_0(3\bar{u}((z+\bar{z})u_0+y_1)^2+((z+\bar{z})u_0+y_1)^3)dx,\\
&=&\displaystyle\lambda_1z-\frac{1}{C}\int_{\Omega}u_03\bar{u}(z+\bar{z})^2u_0^2dx-\frac{6}{C}\int_{\Omega}u_0^2\bar{u}(z+\bar{z})y_1dx-\frac{3}{C}\int_{\Omega}u_0^4z^2\bar{z}dx,\\
&=&\displaystyle\lambda_1z-\frac{3}{C}(z+\bar{z})^2\int_{\Omega}\bar{u}u_0^3dx-\frac{3}{C}z^2\bar{z}\int_{\Omega}u_0^2\bar{u}w_{20}^1dx-\frac{6}{C}z^2\bar{z}\int_{\Omega}u_0^2\bar{u}w_{11}^1dx-\frac{3}{C}z^2\bar{z}\int_{\Omega}u_0^4dx+...
\end{array}
\end{equation}
For equation \eqref{eq:yt}, we only write the terms with order up to $2$,
\begin{equation}
\label{eq:y}
y_t=\mathcal{F}(y)-\frac{3}{\epsilon}(z+\bar{z})^2\begin{pmatrix}
\bar{u}u_0^2\\
0
\end{pmatrix}
+\frac{6}{C}(z+\bar{z})^2 \int_\Omega \bar{u}u_0^3
\begin{pmatrix}
u_0\\
0
\end{pmatrix}
+O(|z|^3).
\end{equation}

We derive \eqref{eq:vc}, we obtain:
\begin{equation}
\label{eq:ytred}
y_t=\lambda_1w_{20}z^2+\lambda_1w_{02}\bar{z}^2+O(|z|^3).
\end{equation}
Identifying with\eqref{eq:y}, we obtain:
\[\left\{\begin{array}{rcl}
(2\lambda_1Id-\mathcal{A})w_{20}&=&H\\
-\mathcal{A}w_{11}&=&H\\
(2\bar{\lambda}_1Id-\mathcal{A})w_{02}&=&H
\end{array}
\right.
\]
with:
\begin{equation*}
H=-\frac{6}{\epsilon}\begin{pmatrix}
\bar{u}u_0^2\\
0
\end{pmatrix}
+\frac{12}{C} \int_\Omega \bar{u}u_0^3
\begin{pmatrix}
u_0\\
0
\end{pmatrix}.
\end{equation*}
This gives,
\[\left\{\begin{array}{rcl}
(2\epsilon\lambda_1-f'(\bar{u})+\frac{1}{2\lambda_1})w^1_{20}-d(w^1_{20})_{xx}&=&\epsilon H^1,\\
w^2_{20}&=&\frac{w^1_{20}}{2\lambda_1},
\end{array}
\right.
\]
\[\left\{\begin{array}{rcl}
w^1_{11}&=&0,\\
w^2_{11}&=&\epsilon H^1,\\
\end{array}
\right.
\]
\[\left\{\begin{array}{rcl}
(2\epsilon\bar{\lambda}_1-f'(\bar{u})+\frac{1}{2\bar{\lambda}_1})w^1_{02}-d(w^1_{02})_{xx}&=&\epsilon H^1,\\
w^2_{02}&=&\frac{w^1_{02}}{2\bar{\lambda}_1}.
\end{array}
\right.
\]
We rewrite in equation \eqref{eq:z}, we obtain:
\begin{equation}
\label{eq:z2}
z_t=\displaystyle \lambda_1z-\big{(}\frac{3}{C}\int_{\Omega}\bar{u}u_0^3dx\big{)}(z+\bar{z})^2-\frac{3}{C}\big{(}\int_{\Omega}u_0^2\bar{u}w_{20}^1dx+\int_{\Omega}u_0^4dx\big{)}z^2\bar{z}+...
\end{equation}
Then, the first Lyapunov coefficient of the Hopf bifurcation is given by:
\begin{align*}
l_1(0)&=\frac{1}{2\omega_0^2}Re(ig_{20}g_{11}+\omega_0g_{21}),
\end{align*}
where,
\begin{align*}
g_{20}&=-\frac{3}{C}\int_{\Omega}\bar{u}u_0^3dx,\\
g_{11}&=-\frac{6}{C}\int_{\Omega}\bar{u}u_0^3dx,\\
g_{21}&=-\frac{3}{C}\big{(}\int_{\Omega}u_0^2\bar{u}w_{20}^1dx+\int_{\Omega}u_0^4dx\big{)},\\
\omega_0&=\frac{1}{\sqrt{\epsilon}}.
\end{align*}
Hence, we find,
\begin{align*}
l_1(0)&=-\frac{3\sqrt{\epsilon}}{2C}(\int_\Omega u_0^4+\int_\Omega\bar{u}u_0^2Re(w^1_{20}))
\end{align*}

\end{proof}

\section{Numerical simulations}
For the numerical simulations, we choose $a=1$ and 
\begin{equation*}
c(x)=p(x^4-2x^2).
\end{equation*}
Numerous methods have been developed to simulate RD systems,  see \cite{Bar90,Dil98,Pre89,Spo00} and references therein cited. Here, we simulate equation \eqref{FHNRD1D} on $(-1,1)$ with an explicit scheme of Runge-Kutta 4 type, with a time step of $10^{-4}$ and a space step of $0.1$. The value of  $\epsilon$ is set to $0.1$.
We obtain:
\begin{itemize}
\item if $p>2.1$, the solution converges toward a stationary solution. The figure \ref{Films} represents $u(x,t)$ for fixed $t=550, 560, 570$ and  $p=2.1$. This solution do not change anymore and has reached the stationary solution. The figure \ref{ev} represents the solution $u(0,t)$ and $u(-1,t)$ for $t\in [500,600]$.
\item If $0<p<2$, we observe periodic solutions. Figure  \ref{Films2} represents the solution $u(x,t)$ for fixed $t$  large enough and $p=2$. Figure \ref{ev2} represents  $u(0,t)$ and $u(-1,t)$ for $t\in [500,600]$.
\end{itemize}
\begin{center}
\begin{figure}
\includegraphics[scale=0.3]{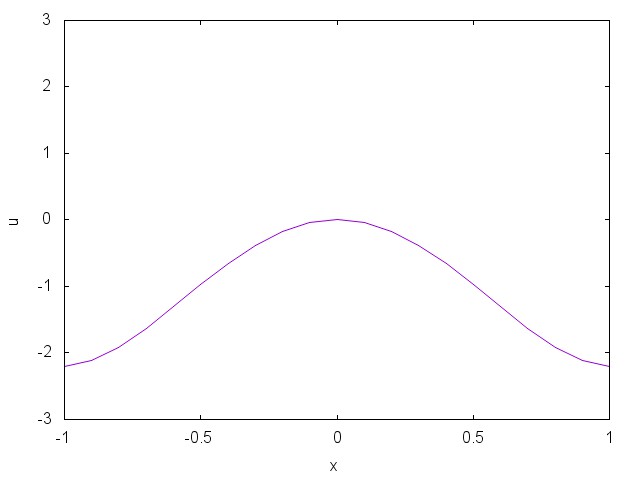}
\includegraphics[scale=0.3]{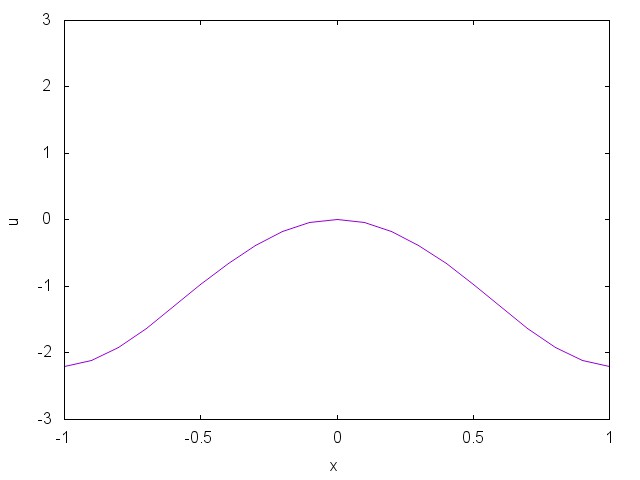}
\includegraphics[scale=0.3]{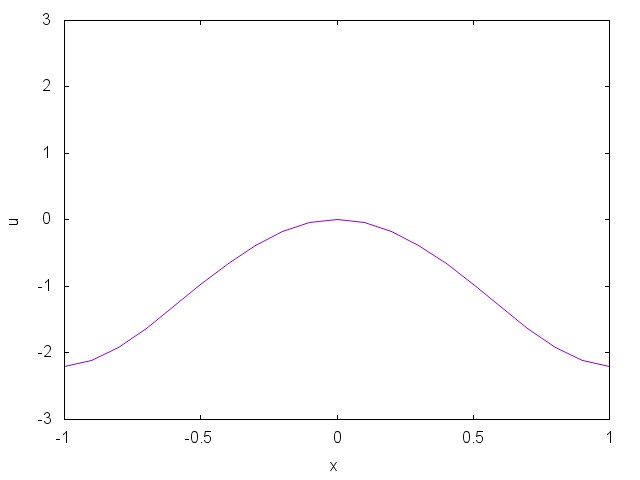}
\caption{Bifurcation between stationary solution and periodic solutions. Here we represent the  solution $u(x,t)$ for $p=2.1$  and  for $t=550, 560, 570$ (from left to right). We observe that the solution  has reached the stationary state $\bar{u}=c(x)=p(x^4-2x).$ }
\label{Films}
\end{figure}
\end{center}
\begin{center}
\begin{figure}
\includegraphics[scale=0.3]{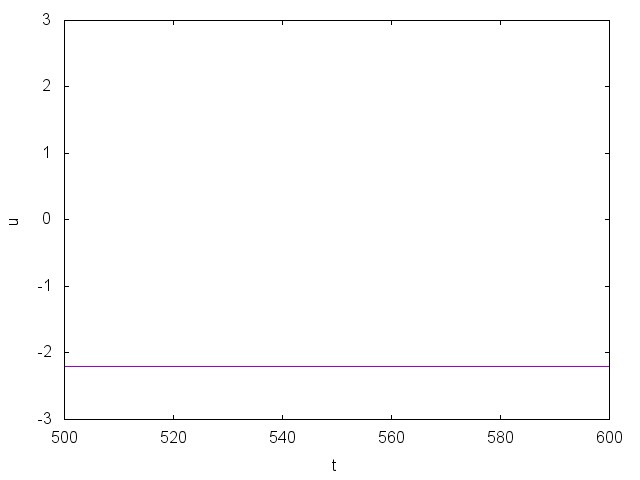}
\includegraphics[scale=0.3]{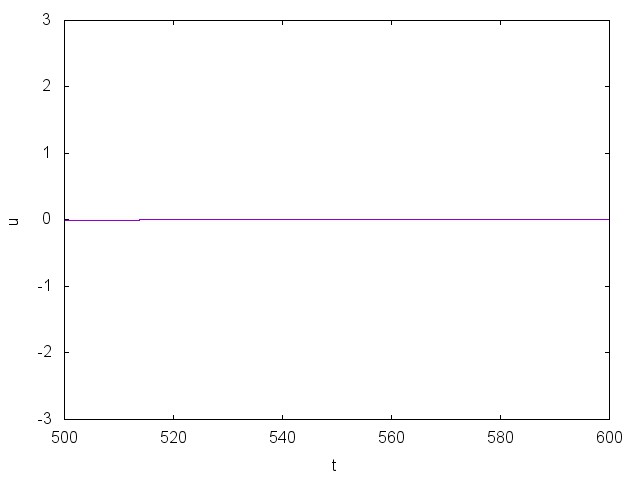} 
\caption{Bifurcation between stationary solution and periodic solutions. Here, we represent the same solution as in figure \ref{Films} but we focus on time evolution at $x=-1$ (left panel) and $x=1$ (right panel). Indeed, we represent $u(-1,t)$ and  $u(0,t)$ for $p=2.1$ and $t\in[500,600]$. We observe that the solution has reached the steady state.}
\label{ev}
\end{figure}
\end{center}

\begin{center}
\begin{figure}
\includegraphics[scale=0.3]{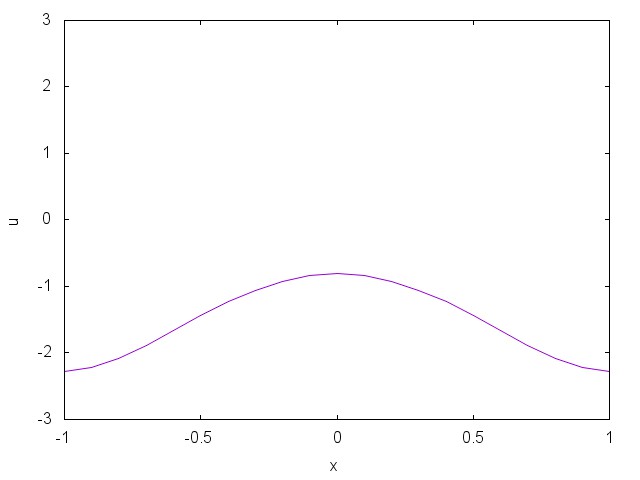}
\includegraphics[scale=0.3]{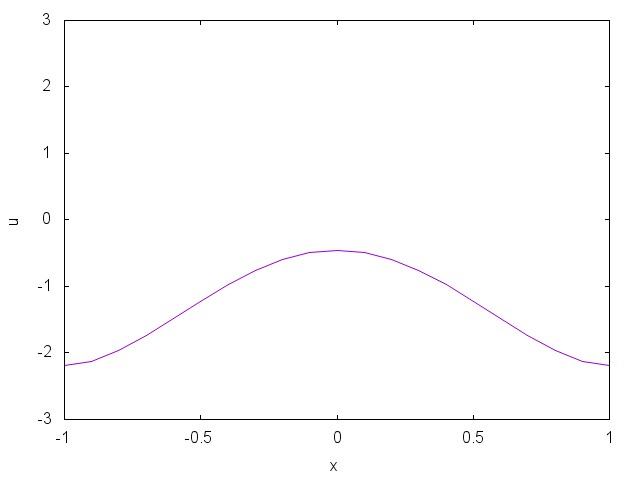}
\includegraphics[scale=0.3]{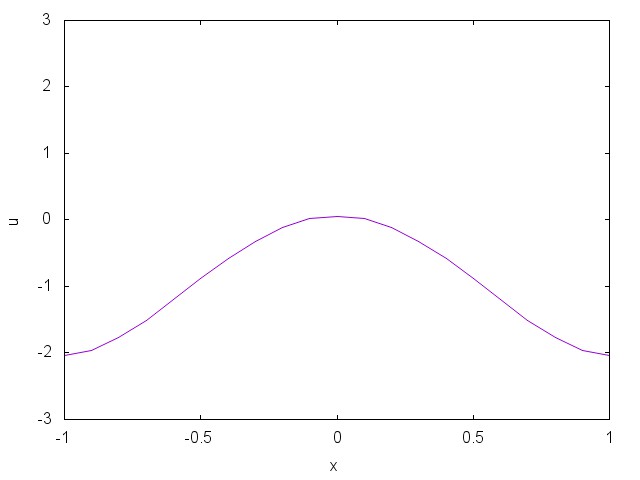}
\caption{Bifurcation between stationary solution and periodic solutions. Here we represent the  solution $u(x,t)$ for $p=2.0$  and  for $t=550, 560, 570$ (from left to right). We observe that the solution  is no longer stationary but it is oscillating. This indicates that the hopf bifurcation parameter value belongs to $(2-2.1)$.}
\label{Films2}
\end{figure}
\end{center}

\begin{center}
\begin{figure}
\includegraphics[scale=0.3]{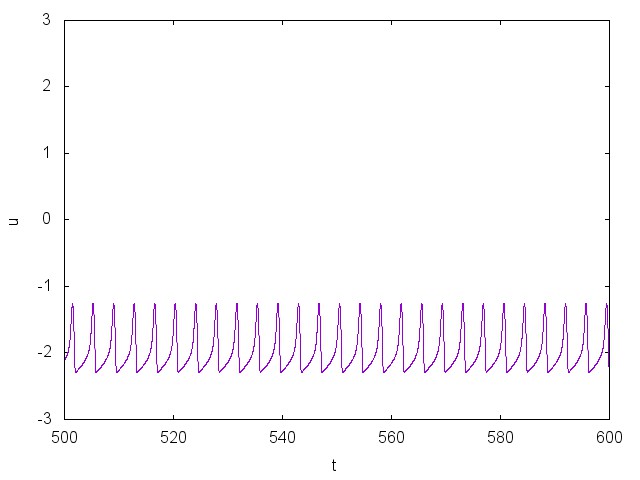}
\includegraphics[scale=0.3]{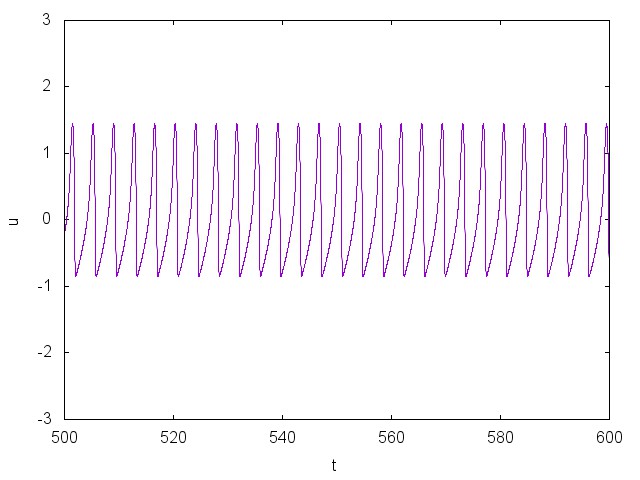} 
\caption{Bifurcation between stationary solution and periodic solutions. Here, we represent the same solution as in figure \ref{Films2} but we focus on time evolution at $x=-1$ (left panel) and $x=1$ (right panel). Indeed, we represent $u(-1,t)$ and  $u(0,t)$ for $p=2.0$ and $t\in[500,600]$. We observe that the solution is oscillating.}
\label{ev2}
\end{figure}
\end{center}
\section{Appendix}
For reader convenience, we give here some results which we have used in the article. 
The following result which proof can be found in \cite{Tes10} gives a result of comparison for solutions of ODEs.
\begin{Theorem}
\label{th:comp}
Assume that $f(x,t)$ is a locally Lipschitz continuous function with respect to $x$
uniformly in $t$. Let $x(t)$ and $y(t)$ be two differentiable functions such that
\begin{equation*}
x(t_0)\leq y(t_0), \quad x'(t)-f(t,x(t)) \leq  y'(t)-f(t, y(t)),\ t\in [t_0, T).
\end{equation*}
Then we have,
\[x(t) \leq y(t),\quad t\in [t_0, T).\]
 Moreover, if 
 \[x(t)<y(t),\] 
 for some $t\in [t_0, T)$,  this remains true for all later times.
\end{Theorem}

\begin{Definition}
\label{def:uppersubsol}
A differentiable function $x^+(t)$ satisfying
\[(x^{+})'(t)> f(t, x^+(t)), t \in [t_0, T),\]
is called an upper solution of equation
\[x'(t)=f(t,x(t)), t \in [t_0, T).\]
 Similarly, a differentiable function $x^-(t)$ satisfying
\[(x^{-})'(t) < f(t, x^-(t)), t \in [t_0, T),\]
is called a lower solution  of equation
\[x'(t)=f(t,x(t)), t \in [t_0, T).\]
\end{Definition}

\begin{Theorem}
\label{th:uppersubsol}
 Let $x^+(t), x^-(t)$ be upper and lower solutions  of the differential
equation $x'= f(t, x)$ on $[t_0, T)$, respectively. Then for every solution $x(t)$ on
$[t_0, T)$, we have
\[x(t) < x^+(t), t\in(t_0, T), \mbox{ whenever } x(t_0)\leq  x^+(t_0),\]
\[x(t) >x^-(t), t\in(t_0, T), \mbox{ whenever } x(t_0)\geq  x^-(t_0).\]
\end{Theorem}

%

\end{document}